\documentclass[12pt,reqno]{amsart}

\setbox0=\hbox{$+$}
\newdimen\plusheight
\plusheight=\ht0
\def\+{\;\lower\plusheight\hbox{$+$}\;}

\setbox0=\hbox{$-$}
\newdimen\minusheight
\minusheight=\ht0
\def\-{\;\lower\minusheight\hbox{$-$}\;}

\setbox0=\hbox{$\cdots$}
\newdimen\cdotsheight
\cdotsheight=\plusheight
\def\cds{\lower\cdotsheight\hbox{$\cdots$}}

\newcommand{\df}{\dfrac}
\newcommand{\f}{\frac}

\renewcommand{\l}{\lambda}

\newcommand{\ti}{\tilde}

\renewcommand{\(}{\left\(}
\renewcommand{\)}{\right\)}

\newcommand{\ra}{\rightarrow}
\renewcommand{\i}{\infty}

\numberwithin{equation}{section} \theoremstyle{plain}
\newtheorem{theorem}{Theorem}[section]
\newtheorem{lemma}[theorem]{Lemma}
\newtheorem{corollary}[theorem]{Corollary}

\begin{document}

\title[Applications of the Heine and Bauer-Muir transformations
 ] {Applications of the Heine and Bauer-Muir transformations
to Rogers-Ramanujan type continued fractions}

\author{ Jongsil Lee, James Mc Laughlin and Jaebum Sohn}

\address{Department of Mathematics, Yonsei University, 50 Yonsei-ro, Seoul, 03722, Korea}
\email{jsglocke@yonsei.ac.kr}
\address{Department of Mathematics, West Chester University, 25 University Avenue, West Chester, PA 19383}
\email{jmclaughlin@wcupa.edu}
\address{Department of Mathematics, Yonsei University, 50 Yonsei-ro, Seoul, 03722, Korea}
\email{jsohn@yonsei.ac.kr.}


\vspace*{0.5in}

\begin {abstract} In this paper we show that various continued fractions for the quotient of general Ramanujan functions $G(aq,b,\l q)/G(a,b,\l)$ may be derived from each other via Bauer-Muir transformations. The separate convergence of numerators and denominators play a key part in showing that the continued fractions and their Bauer-Muir transformations converge to the same limit.

We also show that these continued fractions may be derived from Heine's continued fraction for a ratio of $_2\phi_1$ functions and other continued fractions of a similar type, and by this method derive a new continued fraction for $G(aq,b,\l q)/G(a,b,\l)$.

Finally we derive a number of  new versions of some beautiful continued fraction expansions of Ramanujan for certain combinations of infinite products, with the following being an example:
\begin{multline*}
\frac{(-a,b;q)_{\infty} - (a,-b;q)_{\infty}}{(-a,b;q)_{\infty}+ (a,-b;q)_{\infty}}
= \frac{(a-b)}{1-a b}
\-
\frac{(1-a^2)(1-b^2)q}{1-a b q^2}\\
\-
\frac{(a-bq^2)(b-aq^2)q}{1-a b q^4}
\-
\frac{(1-a^2q^2)(1-b^2q^2)q^3}{1-a b q^6}
\-
\frac{(a-bq^4)(b-aq^4)q^3}{1-a b q^8}
\-
\cds
.
\end{multline*}
\end{abstract}
\maketitle
\renewcommand{\thefootnote}%
             {}

\section{introduction}
Several  known continued fractions of generalized Rogers-Ramanujan type  are known to be equal because they have been shown to converge, at least for certain values of their parameters, to the same ratio of basic hypergeometric series, $G(aq,b,\l q)/G(a,b,\l)$ (see \eqref{gableq1} below). We are justified in terming these continued fractions of \emph{generalized Rogers-Ramanujan type} since they revert back to the Rogers-Ramanujan continued fraction upon setting some of the parameters equal to zero ($a=b=0$ in \eqref{gcf1}, \eqref{gcf2}, \eqref{gcf3} and \eqref{gcf4} below).

One of the results in the present paper is to show directly that these continued fractions are equal, by showing that each one is the Bauer-Muir transformation of one of the others with respect to a certain sequence. Note that in each case it is also shown that the continued fraction and its Bauer-Muir transformation \underline{do} converge to the same limit. In general, this is not straightforward to do, but in the present case the separate convergence of numerators and denominators will make it relatively  easy to show that the continued fractions and their Bauer-Muir transformations converge to the same limit.

It is also shown that certain of these continued fractions for $G(aq,b,\l q)/G(a,b,\l)$ may be derived by specializing the parameters in known continued fraction expansions of certain ratios of $_2\phi_1$ functions (for example, Heine's continued fraction). Similar results were stated for continued fraction expansions of other ratios of $_2\phi_1$ functions in \cite{bhads}, but the proofs in that paper were incomplete, in that the continued fractions were derived by iterating certain three-term recurrences, but the authors failed to address the question of convergence, and did not show that the continued fractions converged to the initial ratio of $_2\phi_1$ functions. Since many famous continued fractions (including the Rogers-Ramanujan continued fraction, Ramanujan's cubic continued fraction, some continued fractions of Gordon) are derived from these known continued fraction expansions  for $G(aq,b,\l q)/G(a,b,\l)$, this shows that these identities ultimately derive from continued fraction expansions of  Heine type.

We also derive a new continued fraction expansion for $G(aq,b,\l q)/G(a,b,\l)$, and use it to derive some new continued fractions for some infinite products and infinite series.
Finally we derive a number of  new versions of some beautiful continued fraction expansions of Ramanujan for certain combinations of infinite products. An example of one of these latter new identities is
\begin{multline*}
\frac{(-a,b;q)_{\infty} - (a,-b;q)_{\infty}}{(-a,b;q)_{\infty}+ (a,-b;q)_{\infty}}
= \frac{(a-b)}{1-a b}
\-
\frac{(1-a^2)(1-b^2)q}{1-a b q^2}\\
\-
\frac{(a-bq^2)(b-aq^2)q}{1-a b q^4}
\-
\frac{(1-a^2q^2)(1-b^2q^2)q^3}{1-a b q^6}
\-
\frac{(a-bq^4)(b-aq^4)q^3}{1-a b q^8}
\-
\cds
.
\end{multline*}

We begin by recalling some notation. A continued fraction
\begin{align}\label{cf}
b_0 +     {\overset {\infty} {\underset{n=1} {K} }} (a_n
/b_n):=b_0+\df{a_1}{b_1}\+\df{a_2}{b_2}\+\df{a_3}{b_3}\+\cds=b_0 +
\cfrac {a_1} {b_1 + \cfrac {a_2} {b_2 + \cfrac {a_3} { b_3 +
\cdots}}},
\end{align}

\noindent where $a_n \neq 0,\ b_n \in \mathbb{C}$ can be regarded as a
composition of M\"obius transformations or linear fractional
transformations,

\begin{align}
S_1 (w)=b_0 +s_1 (w),\quad S_n (w)=S_{n-1} (s_n (w)),\quad n \geq
2,
\end{align}\label{sn}
where
$$s_n (w)=\f{a_n}{b_n +w},\quad n \geq 1.$$
The convergence behavior is well known from  M\"obius
transformations \cite{gill1,pree}.

We say that $b_0+{\overset {\infty} {\underset{n=1} {K} }}(a_n /b_n)$ converges to a value $f \in  {\hat
{\mathbb{C}}}(=\mathbb{C}\cup\infty)$ if its sequence of approximants
$$f_n = b_0 +\df{a_1}{b_1}\+\df{a_2}{b_2}\+\cds\+\df{a_n}{b_n}$$
converges to a limit $f\in  {\hat {\mathbb{C}}} $ as $n \ra \i$.

If ${\overset {\infty} {\underset{n=1} {K} }}(a_n /b_n )$ converges, so do all of its tails

\begin{align}\label{tail}
f^{(n)}={\overset {\infty} {\underset{\nu=n+1} {K} }}\df{a_{\nu}}{b_{\nu}}=\df{a_{n+1}}{b_{n+1}}\+\df{a_{n+2}}{b_{n+2}}\+
\df{a_{n+3}}{b_{n+3}}\+\cds, \quad n \geq 0,
\end{align}

\noindent and we have $f^{(0)}=f$ and we say $\{f^{(n)}\}$ is the
sequence of {\em right tails} \cite{waad2} of ${\overset {\infty} {\underset{n=1} {K} }}(a_n /b_n )$.

If $g_n = f^{(n)}, $ then from \eqref{tail},

\begin{align}\label{tail2}
g_n=\frac{a_{n+1}}{b_{n+1}+g_{n+1}},\quad n \geq 0.
\end{align}

\noindent A sequence $\{g^{(n)}\}$ which satisfies \eqref{tail2}
is said to be a sequence of {\em wrong tails} \cite{waad2} if
$g^{(0)} \neq f^{(0)}$.

We call
\begin{align*}
S_n (w_n ) =b_0 + \df{a_1}{b_1}\+ \df{a_2}{b_2} \+\cds \+
\df{a_n}{b_n +w_n}
\end{align*}

\noindent the $n$th {\em modified approximant} and $w_n $  a {\em
modifying
factor} \cite{gill3}.

On page 41 in Ramanujan's lost notebook \cite{lnb} we find the
following continued fraction. For any complex numbers $a,\ b,\
\l$, and $q$, but with $|q|<1$, define
\begin{align}\label{gableq1}
G(a,b,\l):=G(a,\l;b;q):=\sum_{n \geq 0} \f{q^{(n^2 +n)
/2}(a+\l)\cdots(a+\l q^{n-1})}{(1-q)\cdots(1-q^n
)(1+bq)\cdots(1+bq^n )}.
\end{align}

\noindent Then
\begin{equation}
\f{G(aq,b,\l q)}{G(a,b,\l)}= \f{1}{1}\+ \df{aq+\l q}{1}\+
\df{bq+\l q^2}{1}\+\df{aq^2 +\l q^3 }{1}\+ \df{bq^2 +\l q^4}{1}\+
\cds \label{gcf1}.
\end{equation}

In particular, if we first replace  $q$ by $q^2$ and $a$ by
$q^{-1}$ and then set $ b=1,\ \l=0$ in \eqref{gcf1}, we obtain the
Rogers-Ramanujan continued fraction. Furthermore, setting $b=\l=1$
and $b=0,\ \l=1$ in \eqref{gcf1} give the Ramanujan's cubic and
the G\"ollnitz--Gordon continued fraction, respectively. In this
paper, we mostly investigate these three continued fractions.

On page 43 in his lost notebook \cite{lnb} we find another continued
fraction for quotients of the function $G(a,\l;b;q)$:

\begin{equation}
\f{G(aq,b,\l q)}{G(a,b,\l)}= \f{1}{1}\+\df{aq+\l
q}{1-aq+bq}\+\df{aq+\l q^2}{1-aq+bq^2}\+\df{aq+\l
q^3}{1-aq+bq^3}\+\cds \label{gcf2}.
\end{equation}

Note that, by comparison with the continued fraction in Theorem 2.2 of \cite{BMcLW}, it is necessary to have $|aq|<1$ for
\eqref{gcf2} to hold. Bhargava and Adiga \cite{bhadi} have proved not only
this continued fraction but also the following continued fraction (also stated by Ramanujan on page 43 in his lost notebook \cite{lnb}):

\begin{equation}
\f{G(aq,b,\l q)}{G(a,b,\l)}= \f{1}{1+aq}\+\df{\l q -abq^2
}{1+bq+aq^2 } \+ \df{\l q^2 -abq^4 }{1+bq^2 +aq^3 } \+ \df{\l q^3
-abq^6 }{1+bq^3 +aq^4 } \+\cds
 \label{gcf3}.
\end{equation}

\noindent In 1974, M. Hirschhorn \cite{hirsh1} proved  the
following identity involving the continued fraction in
\eqref{gcf1}. Let $H(a,b,c,x):= \sum_{r=0}^{\i} \f{x^{(r^2
-r)/2}(b+cx)\cdots(b+cx^r )}{(x)_r (a)_{r+1}}$ and $|a|<1$, then

\begin{align*}
1+a+b+\df{cx-a}{1+a+bx}\+\df{cx^2 -a}{1+a+bx^2}\+\cds
=\f{H(a,b,c,x)}{H(a,bx, cx, x)}.
\end{align*}

Now note that
\begin{align*}
\f{G(a,\l;b;q)}{G(aq,\l q;b;q)}
&=\f{\left(\f{1}{1+b}\right)\sum_{n \geq 0} \f{q^{(n^2 +n)
/2}(a+\l)\cdots(a+\l q^{n-1})}{(1-q)\cdots(1-q^n
)(1+bq)\cdots(1+bq^n )}}{\left(\f{1}{1+b}\right)\sum_{n \geq 0}
\f{q^{(n^2 +3n) /2}(a+\l)\cdots(a+\l q^{n-1})}{(1-q)\cdots(1-q^n
)(1+bq)\cdots(1+bq^n )}}\\
&= \f{H(-b,aq,\l ,q)}{H(-b,aq^2 ,\l q,q)}.
\end{align*}

From these we can easily obtain the following identity (which thus holds for $|b|<1$):

\begin{equation*}
1+ \df{aq+\l q}{1}\+ \df{bq+\l q^2}{1}\+\df{aq^2 +\l q^3 }{1}\+
\df{bq^2 +\l q^4}{1}\+ \cds
\end{equation*}

\begin{equation}
  = 1-b+aq +\df{\l
q+b}{1-b+aq^2}\+\df{\l q^2 +b}{1-b+aq^3}\+\df{\l q^3
+b}{1-b+aq^4}\+\cds, \label{gcf4}
\end{equation}

\noindent and we examine this case again in the following section.

\section{Heine's continued fraction and Rogers--Ramanujan type continued fractions}

The $q$--analog of Gauss's continued fraction \cite{gauss} is
called {\em Heine's continued fraction} \cite{heine} and is given
by the quotient of two basic hypergeometric series
{\allowdisplaybreaks
\begin{align}\label{hcf}
\f{ _2 \phi_1 (a,b;c;q;z)}{ _2 \phi_1 (a,bq;cq;q;z)}
&=1+\f{(1-a)(c-b)z}{(1-c)(1-cq)}\+\df{(1-c)(1-bq)(cq-a)z}{1-cq^2}\nonumber\\
&\+ \df{(1-aq)(cq-b)zq}{1-cq^3}\+\df{(1-bq^2 )(cq^2
-a)zq}{1-cq^4}\+\cds \\
&=: 1+ {\overset {\infty} {\underset{n=1} {K} }} (\f{a_n z}{1}),
\end{align}}
where
\begin{eqnarray} \label{hcfa1}
a_{2n+1}=-\f{ q^{n}(1-aq^n )(b-cq^n )}{(1-cq^{2n})(1-cq^{2n+1})},
\quad n \geq 0,
\end{eqnarray}
\begin{eqnarray} \label{hcfa2}
a_{2n}=-\f{ q^{n-1}(1-bq^n )(a-cq^n )}{(1-cq^{2n-1})(1-cq^{2n})},
\quad n \geq 1
\end{eqnarray}
\noindent and basic hypergeometric series $ _2 \phi_1 (a,b;c;q;z)
$ is defined by
$$ _2 \phi_1 (a,b;c;q;z)= \sum_{n=0}^\infty \f{(a;q)_n
(b;q)_n}{(c;q)_n (q;q)_n} z^n. $$

The observation that Ramanujan's identity \eqref{gcf1}  follows from Heine's continued fraction \eqref{hcf} has apparently
not being noticed before. This new proof of \eqref{gcf1}  is simpler and more direct than the proofs of Andrews \cite{A79}, Hirschhorn \cite{H80} and Adiga and Bhargava \cite{bhadi}. In \cite{bhads} the authors give a similar proof that used a continued fraction expansion for a different ratio of $_2\phi_1$ functions, and likewise derived the other known continued fraction  representations of $G(a,b,\l)/G(aq,b,\l q)$ due to Ramanujan and Hirschhorn mentioned above. However, the authors in \cite{bhads} derived their continued fractions by formally iterating the corresponding three-term recurrences, and did not prove convergence to the initial quotient of $_2\phi_1$ functions, nor considered whether any restrictions on the values of the parameters were necessary for the identities to hold. For example, their derivation of \eqref{gcf2} is missing the requirement that $|aq|<1$ is needed for the identity to hold.

\begin{corollary}\label{cor2.1cf}
If $|q|<1$, then
\begin{equation}\label{cor1cf}
\f{G(a,b,\l)}{G(aq,b,\l q)}= 1+ \df{aq+\l q}{1}\+
\df{bq+\l q^2}{1}\+\df{aq^2 +\l q^3 }{1}\+ \df{bq^2 +\l q^4}{1}\+
\cds.
\end{equation}
\end{corollary}
\begin{proof}
In \eqref{hcf}, set $c=0$ and simultaneously replace $a$ with $-\lambda/a$, $b$ with $-\lambda/b$ and $z$ with $qab/\lambda$, so that the continued fraction in \eqref{hcf} becomes the reciprocal of the continued fraction at  \eqref{gcf1}. The left side of \eqref{hcf} becomes
$$
\f{ _2 \phi_1 (-\lambda/a,-\lambda/b;0;q;qab/\lambda)}{ _2 \phi_1 (-\lambda/a,-\lambda q/b;0;q;qab/\lambda)},
$$
and the reciprocal of the left side of \eqref{gcf1} is obtained after applying Jackson's \cite{J10} transformation formula (see also the second formula at \cite[p.\ 14, Eq.\ (1.5.4)]{gasper})
\begin{equation}\label{hjeq}
\sum_{n=0}^{\infty}\frac{(a,b;q)_n}{(c,q;q)_n}z^n =
\frac{(az;q)_{\infty}}{(z;q)_{\infty}}
\sum_{k=0}^{\infty}\frac{(a,c/b;q)_k}{(c,az,q;q)_k}(-bz)^k q^{k(k-1)/2}
\end{equation}
to each of the $_2\phi_1$ functions.
\end{proof}

Thus all the identities which follow from Ramanujan's identity \eqref{gcf1}, such as the Rogers--Ramanujan continued fraction (see also, for example, section 6.2 in \cite{ABRLNI}) thus may be seen to follow from Heine's continued fraction \eqref{hcf}.

The continued fraction in the following theorem, which involves the same ratio of $_2\phi_1$ functions as Heine's continued fraction, is proved by
using the Heine transformation twice, and appears to be new.

\begin{theorem} \label{th2.4} For $|q|,|z|,|c/b| <1$ we have
\begin{align}
\f{_2 \phi_1 (a,b;c;q;z)}{ _2 \phi_1 (a,bq;cq;q;z)}
&=\f{1-bz}{1-c}+\f{(c-abz)(z-1)}{(1-c)(1-bzq)}\+
\df{(1-c)(1-bq)(cq-a)z}{1-bzq^2}\+\nonumber\\
&\quad \quad \df{(c-abzq)(zq-1)q}{1-bzq^3}\+\df{(1-bq^2 )(cq^2
-a)zq}{1-bzq^4}\+\cds. \label{hcf2}
\end{align}
\end{theorem}
\begin{proof}

If we use the second iterate of Heine transformation \cite
[p.~10]{gasper} twice, we obtain
\begin{align}
\f{_2 \phi_1 ( a,b;c;q,z) } {_2 \phi_1 ( a,bq;cq;q,z)}
&=\f{1-bz}{1-c}\cdot \f {_2 \phi_1 (abz/c,b;bz;q,c/b )} {_2 \phi
_1 (abz/c,bq ;bqz;q, c/b )}. \label{eqnstar}
\end{align}
Using \eqref{eqnstar} and \eqref{hcf}, we obtain a continued
fraction equivalent to \eqref{hcf2}
\begin{align*}
(1-c)&\f{_2 \phi_1 ( a,b;c;q,z) } {_2 \phi_1 ( a,bq;cq;q,z)} =
(1-bz)\f {_2 \phi_1 (abz/c,b;bz;q,c/b )} {_2 \phi _1 (abz/c,bq
;bqz;q, c/b )}\\
&=1-bz+\f{(c-abz)(z-1)}{1-bzq}\+\df{(1-bq)(cqz-az)}{1-bzq^2}\+\\
&\qquad \df{(c-abzq)(zq-1)q}{1-bzq^3}\+\df{(1-bq^2 )(cq^2
z-az)q}{1-bzq^4}\+\cds.
\end{align*}
\end{proof}


One implication of this continued fraction is a new continued fraction expansion for the quotient $G(a,b,\l)/G(aq,b,\l q)$.
\begin{corollary}
If $|q|<1$, then
\begin{multline}\label{cor2cf}
\f{G(a,b,\l)}{G(aq,b,\l q)}= 1+aq+
\df{\l q-a b q^2}{1+aq^2}
\+
\df{bq+\l q^2}{1+aq^3}
\+
\df{\l q^3 - a b q^5 }{1+aq^4}
\+
\df{bq^2 +\l q^4}{1+aq^5}\\
\+
\df{\l q^5 - a b q^8 }{1+aq^6}
\+
\df{bq^3+\l q^6}{1+aq^7}
\+
\cds.
\end{multline}
\end{corollary}
\begin{proof}
Make the same substitutions as in Corollary \ref{cor2.1cf}, namely, set $c=0$ and simultaneously replace $a$ with $-\lambda/a$, $b$ with $-\lambda/b$ and $z$ with $qab/\lambda$ in \eqref{hcf2}. The right side of \eqref{hcf2} becomes the right side of \eqref{cor2cf}, while Jackson's identity \eqref{hjeq} once again gives that the left side of \eqref{hcf2} is equal to  $G(a,b,\l)/G(aq,b,\l q)$.
\end{proof}

To apply \eqref{cor2cf} to derive specific new continued fraction identities,  we consider some existing identities in the literature involving the ratio $G(a,b,\lambda)/G(a q,b,\lambda q)​$ where all of the parameters $a$, $b$ and $\lambda$ are non-zero (as the continued fraction in  \eqref{cor2cf} mostly reverts back to known continued fractions otherwise). We gave alternative continued fraction expansions for some functions considered by Ramanujan.

The quantities on the left sides of  \eqref{mod6coreq}, \eqref{ser3cor2eq} and \eqref{ser3cor2eq2} below each appear in other continued fraction identities due to Ramanujan (see, respectively, \cite[p.\ 154, Corollary 6.2.7]{ABRLNI}, \cite[p.\ 155, Corollary 6.2.9]{ABRLNI} and \cite[p.\ 156, Corollary 6.2.11]{ABRLNI}).

\begin{corollary}\label{mod6cor}
If $|q|<1$, then
\begin{multline}\label{mod6coreq}
\frac{(q^3,q^3;q^6)_{\infty}}{(q,q^5;q^6)_{\infty}}
=
1+q
+
\frac{q^2-q^3}{1+q^3}
\+
\frac{q^2+q^4}{1+q^5}
\+
\frac{q^6-q^9}{1+q^7}
\+
\frac{q^4+q^8}{1+q^9}\\
\+
\frac{q^{10}-q^{15}}{1+q^{11}}
\+
\frac{q^6+q^{12}}{1+q^{13}}
\+
\cds.
\end{multline}
\end{corollary}
\begin{proof}
In \eqref{cor2cf} replace $q$ with $q^2$, set $a=1/q$ and $b = \lambda =1$, and right side becomes the right side of \eqref{mod6coreq}. For the left side, with the notation $G(a,b,\l)=G(a,\l;b;q)$ and employing the identities from the Slater list \cite{S52}
\begin{align}\label{mod6prodseq}
G(1/q,1;1;q^2)&=\sum_{n=0}^{\infty}\frac{(-q;q^2)_n q^{n^2}}{(q^4;q^4)_n}
=(q^3,q^3,q^6;q^6)_{\infty}\frac{(-q;q^2)_{\infty}}{(q^2;q^2)_{\infty}}, \\
G(q,q^2;1;q^2)&=\sum_{n=0}^{\infty}\frac{(-q;q^2)_n q^{n^2+2n}}{(q^4;q^4)_n}
=(q,q^5,q^6;q^6)_{\infty}\frac{(-q;q^2)_{\infty}}{(q^2;q^2)_{\infty}},
\end{align}
and the result follows after simplifying the quotient of infinite products. The series-product identities at  \eqref{mod6prodseq} were both stated by Ramanujan, and may also be found in \cite[pp.\ 85, 87, Entry 4.2.7, Entry 4.2.11]{ABRLNII}.
\end{proof}

\begin{corollary}\label{ser3cor}
If $|q|<1$, then
\begin{multline}\label{ser3cor2eq}
\sum_{n=0}^{\infty}(-1)^n q^{n(3n+2)}(1+q^{2n+1})
=
\frac{1}{1-q}
+
\frac{q^2-q^3}{1-q^3}
\+
\frac{q^4-q^2}{1-q^5}
\+
\frac{q^6-q^9}{1-q^7}
\+
\frac{q^8-q^4}{1-q^9}\\
\+
\frac{q^{10}-q^{15}}{1-q^{11}}
\+
\frac{q^{12}-q^6}{1-q^{13}}
\+
\cds.
\end{multline}
\end{corollary}
\begin{proof}
In \eqref{cor2cf}, invert both sides and replace $q$ with $q^2$, set $a=-1/q$ and $b = -1$ and $\lambda =1$, and right side becomes the right side of \eqref{ser3cor2eq}. For the left side, we use the identity (see \cite[pp. 155--156, Corollary 6.2.9]{ABRLNI})
\begin{equation}\label{ser3eq}
\frac{G(-q,q^2;-1;q^2)}{G(-1/q,1;-1;q^2)}
=\sum_{n=0}^{\infty}(-1)^n q^{n(3n+2)}(1+q^{2n+1}),
\end{equation}
and the result again follows.
\end{proof}

\begin{corollary}\label{ser3cor2}
If $|q|<1$, then
\begin{multline}\label{ser3cor2eq2}
1-\sum_{n=1}^{\infty} q^{n(3n-1)/2}(1-q^{n})
=
\frac{2}{2+q}
\+
\frac{q-q^3}{1+q^3}
\+
\frac{q^2+q^3}{1+q^5}
\+
\frac{q^5-q^9}{1+q^7}
\+
\frac{q^4+q^7}{1+q^9}\\
\+
\frac{q^{9}-q^{15}}{1+q^{11}}
\+
\frac{q^{6}+q^{11}}{1+q^{13}}
\+
\cds.
\end{multline}
\end{corollary}
\begin{proof}
The proof again follows similar lines.
In \eqref{cor2cf}, invert both sides and replace $q$ with $q^2$, set $a=1/q$ and $b = 1$ and $\lambda =1/q$.  To get the right side of \eqref{ser3cor2eq2}, invert, add 1, invert again and multiply by 2. For the left side, we again use an identity proved by Andrews and Berndt (see \cite[pp. 156--158, Corollary 6.2.11]{ABRLNI})
\begin{equation}\label{ser3eq22}
\frac{2}{1+\displaystyle{\frac{G(q,q;1;q^2)}{G(1/q,1/q;1;q^2)}}}
=1-\sum_{n=1}^{\infty} q^{n(3n-1)/2}(1-q^{n}),
\end{equation}
and the result once again follows.
\end{proof}

We also make use later of the identities in the following theorem.
\begin{theorem}\label{2phi1cfthm3}
Let $|q|, |z|<1$.

(i) If $|az/q|<1$,  then
\begin{multline}\label{2phi1eq31}
(1-c)\frac{_2\phi_1(a,b;c;q;z)}{_2\phi_1(a,bq;cq;q;z)}= (1-c)+(1-bq/a)az/q\\
-
\frac{(1-cq/a)(1-bq)az/q}{(1-cq)+(1-bq^{2}/a)az/q}
\-
\frac{(1-cq^{2}/a)(1-bq^{2})az/q}{(1-cq^2)+(1-bq^{3}/a)az/q}\phantom{asadasdaadaa}\\
\-
\frac{(1-cq^3/a)(1-bq^3)az/q}{(1-cq^3)+(1-bq^{4}/a)az/q}
\-
\frac{(1-cq^{4}/a)(1-bq^{4})az/q}{(1-cq^4)+(1-bq^{5}/a)az/q}
\-
\cds.
\end{multline}
(ii) If  $|az/q|=1$ but $az/q\not =1$, then
\begin{equation}\label{2phi1eq3105}
(1-c)\frac{_2\phi_1(a,b;c;q;z)}{_2\phi_1(a,bq;cq;q;z)}=
\lim_{n \to \infty}
\frac{A_n -a z A_{n-1}/q}{B_n -a z B_{n-1}/q},
\end{equation}
where $A_n/B_n$ denotes the $n$-th approximant of the continued fraction on the right side of \eqref{2phi1eq31}.

(iii) If $|az/q|>1$, then
{\allowdisplaybreaks\begin{multline}\label{2phi1eq311}
\frac{az}{q}\left(1-\frac{bq}{a}\right)
\frac{_2\phi_1(q/a,c/a;bq/a;q;q/z)}
{_2\phi_1(q/a,cq/a;bq^2/a;q;q/z)}= (1-c)+(1-bq/a)az/q\\
-
\frac{(1-cq/a)(1-bq)az/q}{(1-cq)+(1-bq^{2}/a)az/q}
\-
\frac{(1-cq^{2}/a)(1-bq^{2})az/q}{(1-cq^2)+(1-bq^{3}/a)az/q}\phantom{asadasdaadaa}\\
\-
\frac{(1-cq^3/a)(1-bq^3)az/q}{(1-cq^3)+(1-bq^{4}/a)az/q}
\-
\frac{(1-cq^{4}/a)(1-bq^{4})az/q}{(1-cq^4)+(1-bq^{5}/a)az/q}
\-
\cds.
\end{multline}}
\end{theorem}
A slightly incomplete proof of \eqref{2phi1eq31} may be found in \cite{ABBW85}, and Lorentzen \cite{J89} gave a complete proof of the special case of \eqref{2phi1eq31} and \eqref{2phi1eq311} need to prove Ramanujan's identity \eqref{cframentry12} below. A full proof may be found in \cite[Chapter 16, Theorem 16.17]{McL16}.

Also observe that the continued fractions on the right sides of \eqref{2phi1eq31} and \eqref{2phi1eq311} are the same.

One implication is a proof of Hirschhorn's continued fraction for Ramanujan's ratio $G(a,\lambda;b;q)/G(aq,\lambda q;b;q)$.
\begin{corollary}
If $|q|,\,|b|<1$, then
\begin{equation}\label{gablqhirsch}
\frac{G(a,\lambda;b;q)}{G(aq,\lambda q;b;q)}=
1-b+aq +\dfrac{\lambda
q+b}{1-b+aq^2}\+\dfrac{\lambda q^2 +b}{1-b+aq^3}
\+\cds.
\end{equation}
\end{corollary}
\begin{proof}
The argument follows a similar path as in the proof of \eqref{cor1cf}.
In \eqref{2phi1eq31}, set $c=0$ and simultaneously replace $a$ with $-\lambda/a$, $b$ with $-\lambda/b$ and $z$ with $qab/\lambda$, and after a little manipulation \eqref{gablqhirsch} follows.
\end{proof}

\section{New versions of some continued fraction identities of Ramanujan}

Ramanujan give beautiful continued fraction expansions for two ratios of infinite products. Each of this derive ultimately from continued fraction expansion for the ratio $_2\phi_1(a,b;c;q;z)/\,_2\phi_1(a,b;c;q;z)$ upon specializing the parameters. Since Heine's continued fraction  \eqref{hcf}, the continued fraction at \eqref{hcf2} and that at \eqref{2phi1eq31} all involve the same ratio of $_2\phi_1$ functions, we are able to derive new continued fraction expansions for these infinite products.

The identities in Theorem \ref{2phi1cfthm3} were key to the proof of the first of these, a continued fraction identity of Ramanujan \cite[Second Notebook, Chapter 16, Entry 12]{lnb}. Proofs have been given previously in \cite{ABBW85} and \cite{J89},  but the proof given here is possibly shorter and more direct than either of these. Before coming to this, we prove a necessary lemma. We first recall the Bailey-Daum identity (see \cite[Section 1.8]{gasper}): if $|q/b|$, $|q|<1$, then
\begin{equation}\label{bdeq}
\sum_{k=0}^{\infty}\frac{(a,b;q)_k}{(aq/b,q;q)_k}\left( \frac{-q}{b}\right)^k =
\frac{(-q;q)_{\infty}(aq,aq^2/b^2;q^2)_{\infty}}{(aq/b,-q/b;q)_{\infty}}.
\end{equation}

\begin{lemma}
If $|q|<1$ and $|b|>1$, then
\begin{equation}\label{2phi1prlem}
\frac{_2\phi_1(a,b;a/(bq);q;-b^{-1})}
{_2\phi_1(a,bq;a/b;q;-b^{-1})}
-1=
\frac{1}{1-a/(bq)}
\frac{(a,a/(b^2q);q^2)_{\infty}}
{(aq,a/b^2;q^2)_{\infty}}.
\end{equation}
\end{lemma}
\begin{proof}
By the Bailey-Daum identity \eqref{bdeq},
$$
_2\phi_1(a,bq;a/b;q;-b^{-1}) = \frac{(-q;q)_{\infty}(aq,a/b^2;q^2)_{\infty}}
{(a/b,-1/b;q)_{\infty}},
$$
so what remains to be shown is that
\begin{multline}\label{2phi1diffeq}
_2\phi_1(a,b;a/(bq);q;-b^{-1})-
\, _2\phi_1(a,bq;a/b;q;-b^{-1})\\
=
\frac{(-q;q)_{\infty}(a,a/(b^2q);q^2)_{\infty}}
{(1-a/(bq))(a/b,-1/b;q)_{\infty}}.
\end{multline}
However, after some simple algebra followed by a shift of summation index, the left side of \eqref{2phi1diffeq} simplifies to
$$
\frac{(1-a)(1-a/(b^2q))}{(1-a/(bq))(1-a/b)}
\sum_{n=0}^{\infty} \frac{(aq,bq;q^2)_n}{(aq/b,q;q)_n}
\left (\frac{-1}{b}\right )^n,
$$
and the result  follows after one further application of \eqref{bdeq}.
\end{proof}

We use this lemma to give a new proof of the first of the Ramanujan identities mentioned above.

\begin{corollary}
Let $|q|<1$. Then
\begin{multline}\label{cframentry12}
1-ab
+
\frac{(a-bq)(b-aq)}{(1-ab)(1+q^2)}
\+
\frac{(a-bq^3)(b-aq^3)}{(1-ab)(1+q^4)}
\+
\cds \\
=
\begin{cases}
{\displaystyle
\frac{(a^2q,b^2q;q^4)_{\infty}}{(a^2q^3,b^2q^3;q^4)_{\infty}}}, &|ab|<1,\\
{\displaystyle -ab\frac{(q/a^2,q/b^2;q^4)_{\infty}}{(q^3/a^2,q^3/b^2;q^4)_{\infty}}}, &|ab|>1.
\end{cases}
\end{multline}
\end{corollary}

\begin{proof}
First suppose $|ab|<1$. In \eqref{2phi1eq31},
 replace $q$ with $q^2$ and then replace $a$ with $a^2q$, $b$ with $a/(bq)$, $c$ with $ab$ and $z$ with $-bq/a$.
 The right side of the resulting identity is $1-ab$ plus the left side of \eqref{cframentry12}. Hence the proof of \eqref{cframentry12} will follow if it can be shown that
 \begin{equation}\label{ramprcfeq}
 \frac{(1-ab)\, _2\phi_1(a^2q,a/(bq);ab;q^2;-qb/a)}{\, _2\phi_1(a^2q,aq/b;abq^2;q^2;-qb/a)}-(1-ab)
 =
 \frac{(a^2q,b^2q;q^4)}{(a^2q^3,b^2q^3;q^4)}.
 \end{equation}
However, this follows from \eqref{2phi1prlem}, after replacing $q$ with $q^2$ and then $a$ with $a^2q$ and $b$ with $a/(bq)$.

For the case $|ab|>1$, factor out $-ab$ from the continued fraction on the left side of \eqref{cframentry12}, apply an equivalence transformation with each factor $r_i=-1/ab$, and the result follows from the $|ab|<1$ case, with $1/a$ instead of $a$ and $1/b$ instead of $b$.
\end{proof}
Remark:  The identity is also true when the continued fraction terminates, when $a=bq^{2k+1}$, some $k\in \mathbb{Z}$.

We now give two new continued fraction expansion for the infinite product in the previous corollary, in the case $|ab|<1$ (the case $|ab|>1$ is similar and is omitted).
\begin{corollary}
If $|q|,\, |qb/a|<1$, then
\begin{multline}
\label{cframentry12eq3}
\frac{(a^2q,b^2q;q^4)_{\infty}}{(a^2q^3,b^2q^3;q^4)_{\infty}}
=
\frac{(1-a^2q)(1-b^2q)}{(1-abq^2)}
\+
\frac{(a-bq)(b-aq)q^2}{(1-abq^4)}\\
\+
\frac{(1-a^2q^3)(1-b^2q^3)q^2}{(1-abq^6)}
\+
\frac{(a-bq^3)(b-aq^3)q^4}{(1-abq^8)}
\+
\cds.
\end{multline}
\end{corollary}
\begin{proof}
Replace $q$ with $q^2$ in \eqref{hcf}, multiply both sides by $1-c$ and then replace $a$ with $a^2q$, $b$ with $a/(bq)$, $c$ with $ab$ and $z$ with $-bq/a$ so that the left side of \eqref{hcf} becomes the left side of \eqref{ramprcfeq}, and thus equals the right side of \eqref{ramprcfeq}. The same changes on the right side of \eqref{hcf} leads to the right side of \eqref{cframentry12eq3}.
\end{proof}
The restriction $|qb/a|<1$ is there to ensure that the requirements of \eqref{hcf} are met (in particular, that $|z|<1$), but \eqref{cframentry12eq3} may also hold for values of $a$ and $b$ that do not satisfy this requirement.

\begin{corollary}\label{cramprcfeq22}
If $|q|<1$ and   $|b q|<|a|<1/|b|$, then
\begin{multline}\label{cframentry12eq2}
\frac{(a^2q,b^2q;q^4)_{\infty}}{(a^2q^3,b^2q^3;q^4)_{\infty}}
=
1+ab
-
\frac{(a+bq)(b+aq)}{(1+q^2)}
\+
\frac{(a-bq)(b-aq)q^2}{(1+q^4)}\\
\-
\frac{(a+bq^3)(b+aq^3)q^2}{(1+q^6)}
\+
\frac{(a-bq^3)(b-aq^3)q^4}{(1+q^8)}
\-
\cds.
\end{multline}
\end{corollary}
\begin{proof}
The proof is similar to the proof above.
This time replace $q$ with $q^2$ in \eqref{hcf2}, multiply both sides by $1-c$ and then replace $a$ with $a^2q$, $b$ with $a/(bq)$, $c$ with $ab$ and $z$ with $-bq/a$ so that the left side of \eqref{hcf2} becomes the left side of \eqref{ramprcfeq}, and thus equals the right side of \eqref{ramprcfeq}. The same changes on the right side of \eqref{hcf2} leads to the right side of \eqref{cframentry12eq2}.
\end{proof}
Remark: The restrictions on $a$ and $b$ in the corollary are, as above, to ensure that the requirements of  \eqref{ramprcfeq} and \eqref{hcf2} are met, but it may be the case that these restrictions may be relaxed and \eqref{cframentry12eq2} will still hold.

We now consider the second of Ramanujan's continued fraction identities which were alluded to above.
This identity of Ramanujan \cite[Second Notebook, Chapter 16, Entry 12]{lnb} may be proved by employing Heine's continued fraction identity, as the authors in \cite{ABBW85} did. The proof is short, so we give it for the sake of completeness.
\begin{corollary}
If $|q|,\, |a|<1$, then
\begin{multline}\label{ramcort1eq}
\frac{(-a,b;q)_{\infty} - (a,-b;q)_{\infty}}{(-a,b;q)_{\infty}+ (a,-b;q)_{\infty}}
= \frac{a-b}{1-q}
\-
\frac{(a-bq)(b-aq)}{1-q^3}\\
\-
\frac{(a-bq^2)(b-aq^2)q}{1-q^5}
\-
\frac{(a-bq^3)(b-aq^3)q}{1-q^7}
\-
\cds
.
\end{multline}
\end{corollary}
\begin{proof}
In \eqref{hcf}, replace $q$ with $q^2$ and then replace $a$ with $bq/a$, $b$ with $b/a$, $c$ with $q$ and $z$ with $a^2$.
Then invert both sides and multiply both sides by $a-b$, so that the continued fraction on the right side of the resulting identity is the continued fraction on the right side of \eqref{ramcort1eq}.

The left side of the resulting identity is $B/A$, where
\begin{align}\label{BoverAeq}
A=\, _2\phi_1(bq/a,b/a;q;q^2;a^2),&&B=\frac{a-b}{1-q}\, _2\phi_1(bq/a,bq^2/a;q^3;q^2;a^2).
\end{align}
It is an easy check that
$$
A\pm B =\sum_{n=0}^{\infty}\frac{(b/a;q)_n}{(q;q)_n}(\pm a)^n= \frac{(\pm b;q)_{\infty}}{(\pm a;q)_{\infty}},
$$
where the last equality follows from the $q$-binomial theorem. Hence the left side of \eqref{ramcort1eq} is
$$
\frac{(b;q)_{\infty}/(a;q)_{\infty}-(-b;q)_{\infty}/(-a;q)_{\infty}}
{(b;q)_{\infty}/(a;q)_{\infty}+(-b;q)_{\infty}/(-a;q)_{\infty}}
=
\frac{(A+B)-(A-B)}{(A+B)+(A-B)}
=
\frac{B}{A},
$$
and
\eqref{ramcort1eq} follows.
\end{proof}
As with the infinite product in \eqref{cframentry12}, it is now an easy matter to derive two other continued fraction expansions for the combinations of infinite products in \eqref{ramcort1eq}.

\begin{corollary}
Let $|q|<1$. (i) If $|a^2|, \, |ab/q|<1$, then
\begin{multline}\label{ramcort1eq2}
\frac{(-a,b;q)_{\infty} - (a,-b;q)_{\infty}}{(-a,b;q)_{\infty}+ (a,-b;q)_{\infty}}
= \frac{(a-b)q}{(ab+q)(1-q)}
\-
\frac{(a-bq^2)(b-aq^2)}{(ab+q)(1-q^3)}\\
\-
\frac{(a-bq^4)(b-aq^4)q}{(ab+q)(1-q^5)}
\-
\frac{(a-bq^6)(b-aq^6)q}{(ab+q)(1-q^7)}
\-
\cds
.
\end{multline}
(ii) If $|a^2|, \, |aq/b|<1$, then
\begin{multline}\label{ramcort1eq3}
\frac{(-a,b;q)_{\infty} - (a,-b;q)_{\infty}}{(-a,b;q)_{\infty}+ (a,-b;q)_{\infty}}
= \frac{(a-b)}{1-a b}
\-
\frac{(1-a^2)(1-b^2)q}{1-a b q^2}\\
\-
\frac{(a-bq^2)(b-aq^2)q}{1-a b q^4}
\-
\frac{(1-a^2q^2)(1-b^2q^2)q^3}{1-a b q^6}
\-
\frac{(a-bq^4)(b-aq^4)q^3}{1-a b q^8}
\-
\cds
.
\end{multline}
\end{corollary}

\begin{proof}
Make same changes as in the proof of \eqref{ramcort1eq} are made in respectively, \eqref{2phi1eq31} and \eqref{hcf2} (replace $q$ with $q^2$ and then replace $a$ with $bq/a$, $b$ with $b/a$, $c$ with $q$ and $z$ with $a^2$). In each case the identity is manipulated so that the left side becomes $B/A$, as described at \eqref{BoverAeq} and thus equals the desired combination of infinite products. Upon making the same changes in the corresponding continued fractions in \eqref{2phi1eq31} and \eqref{hcf2}, the continued fractions in \eqref{ramcort1eq2} and \eqref{ramcort1eq3} are produced. The details are left to the reader.
\end{proof}
Remark: as above, the restrictions on $a$ and $b$ are derived from the restrictions on the parameters in \eqref{2phi1eq31} and \eqref{hcf2}, but it may be that one or both of the identities in the previous corollary may hold for values of $a$ and $b$ that lie outside the stated restrictions.

\section{Bauer--Muir transformations on the continued fractions from
quotients of $G(a,\l;b;q)$ }

A {\em Bauer--Muir transformation}  of a continued fraction $b_0 +
{\overset {\infty} {\underset{n=1} {K} }} (a_n /b_n )$  studied by G. Bauer \cite{bauer} and T. Muir
\cite{muir} is a (new) continued fraction whose approximants have
the values
\begin{equation}\label{b1}
S_k (w_k ) := b_0 + \df{a_1}{b_1}\+ \df{a_2}{b_2}\+\cds\+
\df{a_k}{b_k +w_k },\quad k=0,1,2,\dots.
\end{equation}

If
\begin{equation}\label{b2}
\lambda_n := a_n - w_{n-1}(b_n + w_n ) \neq 0, \quad n \geq 1,
\end{equation} then it is given by
\begin{equation}\label{b3}
S_k (w_k) = b_0 + w_0 + \df{\lambda_1}{b_1 + w_1}\+\df{a_1
\lambda_2 /\lambda_1} {b_2 +w_2 - w_0 \lambda_2 /
\lambda_1}\+\df{a_2 \lambda_3 / \lambda_2} {b_3 + w_3 - w_1
\lambda_3 /\lambda_2 }\+ \cds.
\end{equation}
We also call \eqref{b3} a {\em TW-transformation} of \eqref{cf}
with respect to $\{w_n \}$ \cite{lisaj3,lisaj2,lw,tw1}.

\noindent This transform can be used repeatedly, that is, we can
apply the Bauer--Muir transformation repeatedly to the new
continued fraction $\ti{b_i}+\ti{f^{(i)}}$ for the $(i+1)$th
iteration, $i=0,1,2,\dots$, if it satisfies the condition
\eqref{b2} at each step.

If we apply the Bauer--Muir transformation repeatedly to the
continued fraction in  \eqref{gcf1}, then we obtain the three
continued fractions in \eqref{gcf2}--\eqref{gcf4}. For
convenience, we take the reciprocals of the continued fractions in
\eqref{gcf1}--\eqref{gcf3}.

\begin{theorem} \label{gcf}
Suppose $a,b,$ and $\l $ do not vanish simultaneously and $q \neq
0, $ $\l \neq abq^n, -b/q^n$, $-a/q^{n-1},\ n=1, 2,3, \dots$, then
{\allowdisplaybreaks
\begin{align}
&1+ \df{aq+\l q}{1}\+ \df{bq+\l q^2}{1}\+\df{aq^2 +\l q^3
}{1}\+ \df{bq^2 +\l q^4}{1}\+ \cds \label{f1}\\
&= 1+aq+\df{\l q -abq^2 }{1+bq+aq^2 } \+ \df{\l q^2 -abq^4
}{1+bq^2 +aq^3 } \+ \df{\l q^3 -abq^6 }{1+bq^3 +aq^4 } \+\cds
 \label{f2}\\
&= 1-b+aq +\df{\l q+b}{1-b+aq^2}\+\df{\l q^2 +b}{1-b+aq^3}\+\df{\l
q^3 +b}{1-b+aq^4}\+\cds \label{f3}\\
&=1+\df{aq+\l q}{1-aq+bq}\+\df{aq+\l q^2}{1-aq+bq^2}\+\df{aq+\l
q^3}{1-aq+bq^3}\+\cds \label{f4}.
\end{align}
}
\end{theorem}
Remark: Equality with the third continued fraction \eqref{f3} needs $|b|<1$ and equality with the fourth continued fraction \eqref{f4} needs $|aq|<1$.
\begin{proof}
Let
\begin{align*}
G(q):=1+ \df{aq+\l q}{1}\+ \df{bq+\l q^2}{1}\+\df{aq^2 +\l q^3
}{1}\+ \df{bq^2 +\l q^4}{1}\+ \cds.
\end{align*}

First if we choose $w_{2n}=aq^{n+1},\ w_{2n+1}=bq^{n+1},\ n \geq
0, $ as modifying factors, then
 $\l_{2n-1}=\l q^{2n-1}-ab q^{2n} \neq 0$ and
$\l_{2n}=\l q^{2n}-ab q^{2n+1} \neq 0$ for $n \geq 1, $ so the
Bauer--Muir transformation exists, and by \eqref{b3},

\begin{align*}
G(q)&=1+aq+\df{\l q -abq^2 }{1+bq} \+\df{aq^2 +\l q^2 }{1}\+
\df{bq^2 +\l q^3}{1}\+ \df{aq^3 +\l q^4 }{1}\+ \cds.
\end{align*}

Let
$$G^{(1)} (q):=1+ bq+ \df{aq^2 +\l q^2 }{1}\+ \df{bq^2 +\l q^3 }{1}\+\df{aq^3 +\l
q^4 }{1}\+ \cds.$$

\noindent To apply Bauer--Muir transformation to $G^{(1)} (q), $
choose $w_{2n}=aq^{n+2},\ w_{2n+1}=bq^{n+2}$,\\
$\ n \geq 0.$ Since
 $\l_{2n-1}=\l q^{2n}-ab q^{2n+2} \neq 0$ and
$\l_{2n}=\l q^{2n+1}-ab q^{2n+3} \neq 0$ for $n \geq 1$, we have a
Bauer--Muir transformation to obtain
$$G^{(1)} (q)=1+bq+aq^2 + \df{\l q^2 -abq^4 }{1+bq^2} \+\df{aq^3 +\l q^3 }{1}\+ \df{bq^3 +\l q^4}{1}\+
\df{aq^4 +\l q^5 }{1}\+ \cds.$$

\noindent Hence we have
\begin{align*}
G(q) &= 1+aq+\df{\l q -abq^2 }{1+bq} \+\df{aq^2 +\l q^2 }{1}\+
\df{bq^2 +\l q^3}{1}\+ \df{aq^3 +\l q^4 }{1}\+ \cds \\
&= 1 + aq + \f{\lambda q - ab q^2}{G^{(1)}(q)} \\
&= 1 + aq + \f{\lambda q - ab q^2}{1+bq + aq^2 } \+ \f{\lambda q^2
- abq^4}{1+bq^2} \+ \f{aq^3 +\lambda q^3}{1}\+ \f{bq^3 +\lambda
q^4}{1} \+ \cds.
\end{align*}

\noindent Let's repeat this process by defining for $i \ge 2, $
\begin{equation}\label{tail11}
G^{(i)} (q):=1+ bq^i + \df{aq^{i+1} +\l q^{i+1} }{1}\+ \df{bq^{i+1} +\l q^{i+2} }{1}\+\df{aq^{i+2} +\l
q^{i+3} }{1}\+ \cds
\end{equation}
and choose $w^{(i)}_{2n}=aq^{(n+1)+i},\
w^{(i)}_{2n+1}=bq^{(n+1)+i},\ n \geq 0.$ Then
 $\l^{(i)}_{2n-1}=\l q^{(2n-1)+i}-ab q^{2n+2i} \neq 0$ and
$\l^{(i)}_{2n}=\l q^{2n+i}-ab q^{(2n+1)+2i} \neq 0$ for $n \geq 1,
$ so that it leads to the Bauer--Muir transformation to obtain
$$ G^{(i)} (q)=1+bq^i +aq^{i+1}+\df{\l q^{i+1} -abq^{2i+2} }{1+bq^{i+1} } \+\df{a q^{i+2} +\l q^{i+2} }{1}\+
\df{b q^{i+2} +\l q^{i+3}}{1 }\+ \cds.$$

\noindent Therefore, after letting $i \to \infty, $  we deduce
that
\begin{align*}
G(q)=1+aq+\df{\l q -abq^2 }{1+bq+aq^2 } \+ \df{\l q^2 -abq^4
}{1+bq^2 +aq^3 } \+ \df{\l q^3 -abq^6 }{1+bq^3 +aq^4 } \+\cds,
\end{align*}
which gives the continued fraction in \eqref{f2}.

Next, we derive continued fraction in \eqref{f3} by applying
Bauer--Muir transformation to the continued fraction in
\eqref{f2}. If we let $w_n = -bq^n ,\ n \geq 0$, then with $\l_n =
\l q^n +bq^{n-1} \neq 0, $ the Bauer--Muir transformation of the
continued fraction in \eqref{f2} that we call as $H(q) $ is
$$ H(q)=1-b+aq+\df{\l q +b }{1+aq^2 } \+ \df{\l q^2 -abq^3 }{1+bq +aq^3 } \+
\df{\l q^3 -abq^5 }{1+bq^2 +aq^4 } \+\df{\l q^4 -abq^7 }{1+bq^3
+aq^5 } \+\cds.$$

Similarly if we let, for $i \geq 1$,
$$H^{(i)} (q) =1+aq^{i+1}+\df{\l
q^{i+1} -abq^{i+2} }{1+bq +aq^{i+2} } \+ \df{\l q^{i+2} -abq^{i+4}
}{1+bq^{2} +aq^{i+3} } \+ \df{\l q^{i+3} -abq^{i+6} }{1+bq^{3}
+aq^{i+4} } \+\cds$$ then with $w^{(i)}_n =-bq^{n},\ n \geq 0, $
and $\l_n=\l q^{n+i}+b q^{n-1} \neq 0$, we have the Bauer--Muir
transformation to obtain
$$H^{(i)} (q)
=1-b+aq^{i+1}+\df{\l q^{i+1} +b }{1 +aq^{i+2} } \+ \df{\l q^{i+2}
-abq^{i+3} }{1+bq +aq^{i+3} } \+ \df{\l q^{i+3} -abq^{i+5}
}{1+bq^{2} +aq^{i+4} } \+\cds.$$ By the same procedure as before
\begin{align*}
H(q)&= 1+aq+\df{\l q -abq^2 }{1+bq+aq^2 } \+ \df{\l q^2 -abq^4
}{1+bq^2 +aq^3 } \+ \df{\l q^3 -abq^6 }{1+bq^3 +aq^4 } \+\cds \\
&=1-b+aq+\df{\l q +b }{1+aq^2 } \+ \df{\l q^2 -abq^3 }{1+bq +aq^3
} \+ \df{\l q^3 -abq^5 }{1+bq^2 +aq^4 } \+\df{\l q^4 -abq^7
}{1+bq^3 +aq^5 } \+\cds \\
&= 1-b+aq + \df{\l q +b}{1-b+aq^2} \+ \df{\l q^2 +b}{1+ a q^3} \+
\df{\l q^3 - ab q^4}{1+bq + aq^4} \+ \df{\l q^4 - ab q^6}{1+bq^2 +
aq^5}\+\cds.
\end{align*}
\noindent Continue this process to obtain the continued fraction
in \eqref{f3}.

To have the continued fraction in \eqref{f4}, apply Bauer-Muir
transformation repeatedly to the continued fraction in \eqref{f3}.
Choose $w_n =b-aq^{n+1},\ n \geq 0, $ then $\l_n = (\l +a)q \neq
0, $ so the Bauer--Muir transformation of the continued fraction
in \eqref{f3} is
$$ 1+\df{aq+\l q}{1}\+\df{\l q^2+bq}{(1-bq)+aq^2}\+\df{\l q^3
+bq}{(1-bq)+aq^3}\+\cds.$$

Again for $i \geq 1$, let
$$K^{(i)} (q) =1+\df{\l q^{i+1}+bq^i}{1-bq^i+aq^2}\+\df{\l q^{i+2}
+bq^{i}}{1-bq^i+aq^{3}}\+\df{\l q^{i+3}
+bq^{i}}{1-bq^i+aq^{4}}\+\cds.$$

\noindent Then with $w^{(i)}_n=bq^i-aq^{n+1},\ n \geq 0,\
\l^{(i)}_n =\l q^{n+i}+aq^n \neq 0,\ n \geq 1, $ we have the
Bauer--Muir transformation
$$K^{(i)} (q) =1+bq^i-aq +\df{\l
q^{i+1}+aq}{1}\+\df{\l
q^{i+2}+bq^{i+1}}{1-bq^{i+1}+aq^{2}}\+\df{\l q^{i+3}
+bq^{i+1}}{1-bq^{i+1}+aq^{3}}\+\cds. $$

\noindent By following the same procedure as before, we obtain the
continued fraction in \eqref{f4}.

Finally, starting from the continued fraction in \eqref{f4}, we
can find the continued fraction expression in \eqref{f1}. If we
use the Bauer--Muir transformation repeatedly by choosing
$w_n^{(2i)}=0,\ w_n^{(2i+1)}=q^{(2i+1)}-q^{2(n+i+1)},\
n,i=0,1,2,\dots$, we obtain the continued fraction in \eqref{f1}.
\end{proof}

\subsection{Convergence of the continued fractions and their Bauer-Muir transformations to the same limits} It is not automatic that the Bauer-Muir transformation of a continued fraction with respect to a sequence $\{\omega_k\}$ converges, or if it does converge, that it converges to the same limit as the original continued fraction.

However, there is a certain class of continued fractions for which this is easily seen to be the case, provided the sequence $\{\omega_k\}$ is sufficiently well-behaved (for example, if the sequence has a limit different from $-1$). The class in question is the class of continued fractions for which the numerator- and denominator convergents converge separately, and fortunately for the present case, the continued fractions in question fall into that class.

Let the $n$-th approximant of the continued fraction $b_0+{\overset {\infty} {\underset{n=1} {K} }}(a_n/b_n)$ be denoted by $A_n/B_n$ and let the $n$-th approximant of its transformation with respect to the sequence $\{\omega_n\}$ be denoted by $C_n/D_n$. From \eqref{b1} it follows that for $n \geq 0$,
\begin{equation*}
C_n = A_n +\omega_n A_{n-1}, \qquad D_n = B_n +\omega_n B_{n-1}, \qquad
\Longrightarrow \frac{C_n}{D_n} = \frac{A_n +\omega_n A_{n-1}}{B_n +\omega_n B_{n-1}}.
\end{equation*}

Now suppose that the sequences $\{A_n\}$ and $\{B_n\}$ converge separately, i.e.,
\begin{equation*}
\lim_{n\to \infty}A_n = A, \qquad \lim_{n\to \infty}B_n = B, \qquad \Longrightarrow
\lim_{n\to \infty}\frac{A_n}{B_n} = \frac{A}{B},
\end{equation*}
for some complex numbers $A$ and $B$. Suppose further that $\lim_{n\to \infty}\omega_n = \omega \not = -1$, for some complex number $\omega$. Then
\begin{equation*}
\lim_{n\to \infty}\frac{C_n}{D_n} =
\lim_{n\to \infty}
\frac{A_n +\omega_n A_{n-1}}{B_n +\omega_n B_{n-1}}=
\frac{A+\omega A}{B+\omega B}=
\frac{A}{B}.
\end{equation*}

Thus the continued fraction $b_0+{\overset {\infty} {\underset{n=1} {K} }}(a_n/b_n)$  and  its Bauer-Muir transformation with respect to the sequence $\{\omega_n\}$ converge to the same limit.

Remark: The case $A=B=0$ but $\lim_{n\to \infty}A_n/B_n$ exists as a number in $\mathbb{C}$ (the latter of course will happen if the continued fraction converges to a number in $\mathbb{C}$) needs a little more care, but can still be dealt with (since $\lim_{n\to \infty}A_n/B_n$ exists, some tail of the sequence $\{B_n\}$ must have all terms non-zero).

As regards the numerators and denominators in the continued fractions in Theorem \ref{gcf} converging separately (employing the notation just above):\\
- Hirschhorn showed in \cite{H80} for \eqref{f1} that $A_n \to (-bq;q)_{\infty} G(a,b,\lambda)$ and that\\ $B_n \to (-bq;q)_{\infty} G(aq,b,\lambda q)$, as $n \to \infty$;\\
- that the numerators and denominators converge separately for the continued fraction in  \eqref{f2} follows from part (iii) in Theorem 3 of \cite{BMcLW};\\
- separate convergence for the continued fraction at \eqref{f3} (in the case $|b|<1$) follows from Hirschhorn's formulas for his $P_{\infty}$ and $Q_{\infty}$ in \cite{hirsh1};\\
- lastly, separate converge for the continued fraction at \eqref{f4} (in the case $|aq|<1$)
is a consequence of  part (iii) in Theorem 2 of \cite{BMcLW}.

To prove that two continued fractions are equal by applying \emph{infinitely} many Bauer-Muir transformations, a little more care is needed. Suppose it desired to show that
$$
b_0 +{\overset {\infty} {\underset{n=1} {K} }}\ \frac{a_n}{b_n}=d_0 +{\overset {\infty} {\underset{n=1} {K} }}\ \frac{c_n}{d_n},
$$
by applying infinitely many many Bauer-Muir transformations, as is the case in each of the continued fraction identities proved in Theorem \ref{gcf}.  Let $C_m/D_m$ denote the $m$-th approximant of  $d_0 +{\overset {\infty} {\underset{n=1} {K} }}(c_n/d_n)$ and suppose that $\lim_{m \to \infty}C_m=C$ and $\lim_{m \to \infty}D_m=D$.
Suppose further that, after the application of $n$ such transformations, it has been shown that
$$
b_0 +{\overset {\infty} {\underset{n=1} {K} }}\ \frac{a_n}{b_n}=d_0+\frac{c_1}{d_1}
\+
\frac{c_2}{d_2}
\+
\cds
\+
\frac{c_{n}}{d_{n}^{(n)}}
\+
\frac{c_{n+1}^{(n)}}{d_{n+1}^{(n)}}
\+
\frac{c_{n+2}^{(n)}}{d_{n+2}^{(n)}}
\+
\cds
.
$$
 Continue the transformation process by applying the next Bauer-Muir transformation with respect to the sequence
 $\{\omega_k^{(n)}\}$ to the tail of the above continued fraction that starts with $d_{n}^{(n)}$. For each $k \geq 1$ define
 $$
 \frac{C_k^{(n)}}{D_k^{(n)}}:=
 -d_n+d_{n}^{(n)}
 +
 \frac{c_{n+1}^{(n)}}{d_{n+1}^{(n)}}
\+
\frac{c_{n+2}^{(n)}}{d_{n+2}^{(n)}}
\+
\cds
\+
\frac{c_{n+k}^{(n)}}{d_{n+k}^{(n)}}.
 $$

Define
$$
 f_{n,k}:=
 -d_n+d_{n}^{(n)}
 +
 \frac{c_{n+1}^{(n)}}{d_{n+1}^{(n)}}
\+
\frac{c_{n+2}^{(n)}}{d_{n+2}^{(n)}}
\+
\cds
\+
\frac{c_{n+k}^{(n)}}{d_{n+k}^{(n)}+ \omega_k^{(n)}}
=\frac{C_k^{(n)}+ \omega_k^{(n)}C_{k-1}^{(n)}}{D_k^{(n)}+ \omega_k^{(n)}D_{k-1}^{(n)}}
 .
 $$

 Suppose that $\lim_{k \to \infty}C_k^{(n)}=C^{(n)}$, $\lim_{k \to \infty}D_k^{(n)}=D^{(n)}$ and $\lim_{k \to \infty}\omega_k^{(n)}=\omega^{(n)}\not = -1$, so that
 $$
f_n:= \lim_{k \to \infty}f_{n,k}
=\lim_{k \to \infty}
\frac{C_k^{(n)}+ \omega_k^{(n)}C_{k-1}^{(n)}}{D_k^{(n)}+ \omega_k^{(n)}D_{k-1}^{(n)}}
=\frac{C^{(n)}}{D^{(n)}}
$$
exists, and
$$
b_0 +{\overset {\infty} {\underset{n=1} {K} }}\ \frac{a_n}{b_n}=d_0+\frac{c_1}{d_1}
\+
\frac{c_2}{d_2}
\+
\cds
\+
\frac{c_{n}}{d_{n}+f_n}
=\frac{C_n+f_n C_{n-1}}{D_n+f_n D_{n-1}}
.
$$

If $\lim_{n \to \infty}f_n=f\not = -1$, then
$$
b_0 +{\overset {\infty} {\underset{n=1} {K} }}\ \frac{a_n}{b_n}=
\lim_{n\to \infty}
\frac{C_n+f_n C_{n-1}}{D_n+f_n D_{n-1}}
=
\frac{C}{D}=
d_0 +{\overset {\infty} {\underset{n=1} {K} }}\ \frac{c_n}{d_n},
$$
as was desired to be shown.

To see how this applies to Theorem \ref{gcf}, we examine the proof that the continued fraction at  \eqref{f1} equals the continued fraction \eqref{f2}.
By the remarks above, the numerator- and denominator convergents converge separately for each continued fraction.
From \eqref{tail11},
\begin{multline}\label{tail22}
-d_n+d_{n}^{(n)}
 +
 \frac{c_{n+1}^{(n)}}{d_{n+1}^{(n)}}
\+
\frac{c_{n+2}^{(n)}}{d_{n+2}^{(n)}}
\+
\cds \\
=- aq^n + \df{aq^{n+1} +\l q^{n+1} }{1}\+ \df{bq^{n+1} +\l q^{n+2} }{1}\+\df{aq^{n+2} +\l
q^{n+3} }{1}\+ \cds
.
\end{multline}
Apart from the initial term, this continued fraction is \eqref{f1} with $a$ replaced with $aq^n$, $b$ with $bq^n$ and $\l$ with $\l q^n$, hence the numerators and denominators also converge separately ($C^{(n)}$ and $D^{(n)}$ from above exist). From the remarks following \eqref{tail11}, $w^{(n)}_{2k}=aq^{(k+1)+n},\
w^{(n)}_{2k+1}=bq^{(k+1)+n},\ k \geq 0,$ and so $\lim_{k \to \infty}\omega_k^{(n)}=\omega^{(n)}=0$. Thus
$f_n = C^{(n)}/D^{(n)}$. From \eqref{tail22} and what has just been said,
$$
f=\lim_{n \to \infty}f_n= \lim_{n \to \infty}\frac{C^{(n)}}{D^{(n)}}=0,
$$
and thus the requirements for the continued fractions at \eqref{f1} and \eqref{f2} to be equal are met.

The arguments for the other pairs of continued fractions being equal are similar.

\section*{Acknowledgement}
The second author's research was partially supported by a grant from the Simons Foundation (\#209175 to James Mc Laughlin).

The third author's research was supported by Basic Science Research Program through the National Research Foundation of
Korea(NRF) funded by the Ministry of Education, Science and Technology (2011-0011257).

{\allowdisplaybreaks

}

\end{document}